\documentclass[12pt,reqno,twoside]{amsart}

\usepackage[T1]{fontenc}
\usepackage[utf8]{inputenc}

\usepackage[a4paper,inner=1in,outer=1in,top=1.2in,bottom=1.4in]{geometry}

\usepackage{amsmath,amssymb,amsthm,mathtools}
\usepackage{bm}
\usepackage{mathrsfs}
\usepackage{dsfont}

\usepackage{graphicx}
\usepackage{tikz}
\usepackage{caption}
\usepackage{subcaption}
\usepackage{enumitem}
\usepackage{cite}

\usepackage[
  colorlinks=true,
  linkcolor=blue!50!black,
  citecolor=green!40!black,
  urlcolor=magenta!60!black
]{hyperref}

\usepackage[nameinlink,capitalize]{cleveref}

\numberwithin{equation}{section}

\theoremstyle{plain}
\newtheorem{theorem}{Theorem}[section]
\newtheorem{lemma}[theorem]{Lemma}

\theoremstyle{definition}

\theoremstyle{remark}
\newtheorem{remark}[theorem]{Remark}
\newtheorem{example}[theorem]{Example}

\title[Semilinear Elliptic Inequalities on Weighted Graphs]
{A Capacitary Approach to Semilinear Elliptic Inequalities with Potentials on Weighted Graphs}

\author[M. Jleli]{Mohamed Jleli}
\address{Mohamed Jleli:  Department of Mathematics, College of Science, King Saud University, Riyadh 11451, Saudi Arabia}
\email{jleli@ksu.edu.sa}

\author[B. Samet]{Bessem Samet}
\address{Bessem Samet: Department of Mathematics, College of Science, King Saud University, Riyadh 11451, Saudi Arabia}
\email{bsamet@ksu.edu.sa}

\date{}



\begin{document}

\begin{abstract}
We develop a capacitary approach to semilinear elliptic inequalities on weighted
graphs with a potential. More precisely, we study the nonexistence of nontrivial
nonnegative solutions of
\[
\Delta u+w(x)u+v(x)u^\sigma\le0
\qquad\text{in }V,
\]
where \((V,\omega,\mu)\) is a connected, locally finite weighted graph,
\(\Delta\) is the associated graph Laplacian, \(\sigma>1\), \(v>0\), and
\(w\) is a real-valued potential. The potential term is handled by means of a
positive solution \(H\) of \(\Delta H+wH=0\), which transforms the operator
\(\Delta+w\) into the \(H\)-Laplacian associated with a new weighted graph. Our
main nonexistence criterion is formulated directly in terms of cut-off
functions and the regions where their \(H\)-Laplacian is controlled. Unlike
metric criteria based on pseudo-metric annuli, our formulation determines the
capacitary sets from the support of the \(H\)-Laplacian estimates for the
cut-off functions. We provide an example showing that our result applies in
situations not covered by previous nonexistence criteria based on structural
lower bounds or pseudo-metric annular volume estimates. We also show that the
growth exponent in our capacitary condition is sharp by constructing an example
for which the condition fails by an arbitrary power \(R^\varepsilon\), while a
positive nontrivial solution exists.
\end{abstract}



\maketitle

\section{Introduction}

In this work, we study semilinear elliptic inequalities on weighted graphs.
More precisely, we consider problems of the form
\begin{equation}\label{P}
\Delta u+w(x)u+v(x)u^\sigma\le 0
\qquad \text{in } V,
\end{equation}
where \((V,\omega,\mu)\) is a connected, locally finite weighted graph,
\(\Delta\) is the associated graph Laplacian, \(w\) is a real-valued potential,
\(v\) is a positive weight, and \(\sigma>1\). Our main purpose is to establish
nonexistence results for nontrivial nonnegative solutions of \eqref{P} and to
investigate the sharpness of the obtained conditions.

The question of nonexistence for elliptic inequalities has a long history. In
the Euclidean setting, a prototype problem corresponding to \(w\equiv0\) is
\begin{equation}\label{P-E}
\Delta u+v(x)u^\sigma\le 0
\qquad\text{in } \mathbb R^N.
\end{equation}
In the particular case \(v\equiv1\) and \(N\ge3\), it is well known that the
only nonnegative solution of \eqref{P-E} is the trivial one whenever
\[
1<\sigma\le \frac{N}{N-2}.
\]
This result goes back to the classical Liouville-type theorems of Gidas and
Gidas--Spruck \cite{Gidas1980,GidasSpruck1981}. 
A major development in this direction was the nonlinear capacity method of
Mitidieri and Pohozaev~\cite{MitidieriPohozaev1998,MitidieriPohozaev1999,MitidieriPohozaev2004}. Their approach relates integral estimates obtained from
suitable test functions to capacitary quantities and leads to nonexistence
criteria for large families of elliptic inequalities in \(\mathbb R^N\). 
In particular, they proved that \eqref{P-E} admits no nontrivial nonnegative
solution, provided that
\[
\liminf_{R\to\infty}
R^{-\frac{2\sigma}{\sigma-1}}
\int_{B_{\sqrt{2}R}\setminus B_R}
v^{-\frac{1}{\sigma-1}}\,dx
<\infty.
\]

Elliptic inequalities have also been studied on Riemannian manifolds, where the
geometry of the underlying space plays a central role. In this setting, volume
growth, curvature assumptions, and capacity estimates provide natural substitutes
for the Euclidean dimension. Several nonexistence results have been obtained for
semilinear elliptic inequalities on complete noncompact manifolds, often by
combining test-function arguments with geometric information on geodesic balls
and annuli. In particular, for the manifold counterpart of \eqref{P-E}, we refer to \cite{GrigoryanKondratiev2010,GrigoryanSun2014,GrigoryanSunVerbitsky2020,MastroliaMonticelliPunzo2015,Sun2014}.   Further results on evolution inequalities in the setting of Riemannian manifolds can be found in \cite{GuSunXiaoXu2020,JleliSamet2023,JleliSametVetro2023,MastroliaMonticelliPunzo2017,MonticelliPunzoSquassina2020,WeiZhang2024}.

Recently, the study of elliptic inequalities has also been extended to weighted graphs; see, for instance,\cite{BiagiMeglioliPunzo2024,GrigoryanLinYang2016,GuHuangSun2023,HaoSun2023,HuangKellerMasamuneWojciechowski2013,MeglioliPunzo2025,MinhQuyetDuong2025,MonticelliPunzoSomaglia2025}. 
 This discrete setting provides a natural framework in which analytic features of elliptic equations interact with the combinatorial and metric structure of the underlying graph. In this context, volume growth, intrinsic metrics, and capacity-type quantities play roles analogous to those played by geodesic balls and capacities on Riemannian manifolds. Nonexistence results for semilinear elliptic inequalities on weighted graphs have therefore been developed under suitable geometric and capacitary assumptions.

The case \(v\equiv1\) and \(w\equiv0\) in \eqref{P} was studied in
\cite{GuHuangSun2023} under the structural condition
\begin{equation}\label{p0-cd}
\frac{\omega_{xy}}{\mu(x)}\ge \frac{1}{p_0},
\end{equation}
for some \(p_0>1\) and for all adjacent vertices \(x,y\in V\). In particular,
it was proved that, if
\[
\mu(B_R(x_0))
=
O\left(R^{\frac{2\sigma}{\sigma-1}}(\ln R)^{\frac{1}{\sigma-1}}\right)
\qquad\text{as } R\to\infty,
\]
then every nonnegative solution of
\[
\Delta u+u^\sigma\le0
\qquad\text{in } V
\]
is identically zero. Here \(B_R(x_0)\subset V\) denotes the ball with center
\(x_0\in V\) and radius \(R>0\), defined with respect to the graph distance,
and
\[
\mu(B_R(x_0))=\sum_{x\in B_R(x_0)}\mu(x).
\]

In \cite{MonticelliPunzoSomaglia2025}, the case \(w\equiv0\) in \eqref{P},
namely,
\begin{equation}\label{P0}
\Delta u+v(x)u^\sigma\le0
\qquad\text{in } V,
\end{equation}
was studied. More precisely, the authors did not require condition
\eqref{p0-cd} and worked with a general pseudo-metric on \(V\), rather than
restricting the analysis to the graph distance. Their nonexistence result was
obtained under an upper bound for the Laplacian of the distance function
\(d(\cdot,x_0)\), where \(x_0\in V\) is fixed, together with a weighted volume
growth condition on annuli centered at \(x_0\).

For our purposes, the nonexistence result stated in Theorem~3.1 of
\cite{MonticelliPunzoSomaglia2025} can be summarized as follows.

\begin{theorem}\label{T-Punzo}
Assume that the following conditions hold:
\begin{itemize}
    \item[{\rm(H1)}] \((V,\omega,\mu)\) is a connected, locally finite, undirected
    weighted graph satisfying
    \[
    \sum_{y\sim x}\omega_{xy}\le C\mu(x),
    \qquad x\in V,
    \]
    for some constant \(C>0\);

    \item[{\rm(H2)}] there exists a pseudo-metric \(d\) on \(V\) with finite
    jump size, that is,
    \[
    j:=\sup_{x\sim y} d(x,y)<\infty,
    \]
    and such that every ball
    \[
    B_r(x):=\{y\in V:\ d(x,y)<r\}
    \]
    is finite for every \(x\in V\) and \(r>0\);

    \item[{\rm(H3)}] there exist \(x_0\in V\), \(R_0>1\), \(\alpha\in[0,1]\),
    and \(C>0\) such that
    \[
    \Delta d(\cdot,x_0)(x)
    \le
    \frac{C}{d^\alpha(x,x_0)}
    \qquad\text{for all } x\in V\setminus B_{R_0}(x_0);
    \]

    \item[{\rm(H4)}] the weighted annular volume condition
    \[
    \sum_{x\in B_{2R}(x_0)\setminus B_R(x_0)}
    v^{-\frac{1}{\sigma-1}}(x)\mu(x)
    =
    O\left(R^{\frac{(1+\alpha)\sigma}{\sigma-1}}\right)
    \qquad\text{as } R\to\infty
    \]
    holds.
\end{itemize}
Then the only nonnegative solution of \eqref{P0} is \(u\equiv0\) in \(V\).
\end{theorem}

This result is both natural and useful. It extends the classical capacitary
philosophy for semilinear elliptic inequalities to the setting of weighted
graphs and gives a clear geometric sufficient condition for nonexistence. In
particular, the annular quantity in {\rm(H4)} plays the role of a capacitary
quantity at infinity, while the pseudo-metric assumptions {\rm(H2)}--{\rm(H3)}
provide a convenient way to construct cut-off functions adapted to the graph.

However, this metric formulation may be restrictive for certain graphs. Indeed,
the proof in \cite{MonticelliPunzoSomaglia2025} uses test functions of the form
\[
\psi\left(\frac{d(x,x_0)}{R}\right),
\]
and the finite jump-size condition leads to estimates supported in enlarged
pseudo-metric annuli such as
\[
B_{2R+j}(x_0)\setminus B_{R-j}(x_0).
\]
The condition {\rm(H4)} then controls these estimates by imposing a weighted
volume bound on metric annuli. Consequently, vertices with large measure may be
included in the annular quantity even though they do not play an essential role
in the variation of the test function. In such situations, the annular volume
condition may fail, even though one can still construct cut-off functions whose
Laplacian is supported on a much smaller part of the graph; see
Section~\ref{sec5}.

In this work, we develop a capacitary approach to \eqref{P} that does not rely on a prescribed pseudo-metric structure. Our method is based on a positive solution \(H\) of the linear equation \[ \Delta H+wH=0 \qquad\text{in }V, \] which allows us to absorb the potential term into a transformed weighted graph. More precisely, writing \(u=Hz\), the operator \(\Delta+w\) is converted into the \(H\)-Laplacian \(\Delta_H\), and the study of \eqref{P} is reduced to a weighted inequality involving \(\Delta_H\) (see Subsection~\ref{subs-Delta-H}). The main assumption is then expressed directly in terms of cut-off functions: we require a family \(\varphi_R\) for which the positive part of \(-\Delta_H\varphi_R\) is controlled on suitable sets \(E_R\). This formulation does not prescribe the sets \(E_R\) through metric annuli.
Rather, they are chosen according to the region where the \(H\)-Laplacian of
the cut-off function is effectively supported. This makes the criterion
applicable to graphs for which annular volume conditions are too restrictive.

We provide an explicit example showing that this flexibility is effective. More
precisely, we construct a weighted graph for which our nonexistence theorem
applies, whereas the result of Gu, Huang, and Sun \cite{GuHuangSun2023} does
not apply because the structural condition \eqref{p0-cd} fails. The same
example also falls outside the scope of Theorem~\ref{T-Punzo}, since the
annular volume condition {\rm(H4)} fails for every admissible pseudo-metric of
the type considered there. This illustrates that the present capacitary
formulation is not merely a reformulation of earlier metric criteria, but gives
a genuinely more flexible nonexistence principle. 
Moreover, we investigate the sharpness of the capacitary growth condition
appearing in our main theorem. We show that the exponent in this condition is
optimal by constructing an example for which the structural assumptions of the
theorem remain valid, but the capacitary quantity grows faster than the
critical rate by an arbitrary factor \(R^\varepsilon\), and a positive
nontrivial solution exists. This shows that the capacitary growth condition
cannot, in general, be weakened.

The rest of the paper is organized as follows.  In Section~\ref{sec2}, we recall the basic notions on weighted graphs, including the difference operator, the graph Laplacian, and the transformed graph Laplacian associated with a positive function \(H\). In Section~\ref{sec3}, we state the main nonexistence result. Its proof is provided in Section~\ref{sec4}. In Section~\ref{sec5}, we present a comparison example with \(w\equiv0\), showing that our result applies in situations not covered by previous annular-volume criteria. Finally, in Section~\ref{sec6}, we discuss the sharpness of the capacitary condition.

\section{Preliminaries}\label{sec2}

\subsection{The weighted graph setting}

Let \(V\) be a countably infinite set. A weighted graph on \(V\) is a triple
\[
(V,\omega,\mu),
\]
where \(\omega:V\times V\to[0,\infty)\) satisfies
\[
\omega_{xy}=\omega_{yx},\qquad \omega_{xx}=0,
\qquad x,y\in V,
\]
and \(\mu:V\to(0,\infty)\). The quantity \(\omega_{xy}\) is called the edge
weight between \(x\) and \(y\), and \(\mu(x)\) is the measure of the vertex
\(x\).

Two distinct vertices \(x,y\in V\) are said to be adjacent, and we write
\(x\sim y\) if
\[
\omega_{xy}>0.
\]
A path joining \(x\) to \(y\) is a finite sequence of vertices
\[
x=x_0,x_1,\ldots,x_m=y
\]
such that
\[
x_{i-1}\sim x_i,\qquad i=1,\ldots,m.
\]
The graph is said to be connected if any two vertices can be joined by a path.
It is said to be locally finite if, for every \(x\in V\), the set
\[
\{y\in V:\ y\sim x\}
\]
is finite.

For every \(B\subset V\), we set
\[
\mu(B):=\sum_{x\in B}\mu(x).
\]

\subsection{Difference and Laplacian operators}

Let \((V,\omega,\mu)\) be a connected, locally finite weighted graph. For a
function \(f:V\to\mathbb R\), we define the difference operator by
\[
\nabla_{xy}f=f(y)-f(x),
\qquad x,y\in V.
\]
The Laplacian operator associated with \((V,\omega,\mu)\) is defined by
\[
\Delta f(x)
=
\frac{1}{\mu(x)}
\sum_{y\sim x}\omega_{xy}\bigl(f(y)-f(x)\bigr),
\qquad x\in V.
\]
Equivalently,
\[
\Delta f(x)
=
\frac{1}{\mu(x)}
\sum_{y\sim x}\omega_{xy}\nabla_{xy}f,
\qquad x\in V.
\]
Because the graph is locally finite, the above sums are finite for every
\(x\in V\).

For \(f:V\to\mathbb R\), we denote its support by
\[
\operatorname{supp} f=\{x\in V:\ f(x)\ne 0\}.
\]
We say that \(f\) has finite support if \(\operatorname{supp} f\) is a finite
subset of \(V\).

We shall use the following integration by parts formula. Assume that at least
one of the functions \(f,g:V\to\mathbb R\) has finite support. Then
\begin{equation}\label{Int-parts}
\sum_{x\in V}\Delta f(x)g(x)\mu(x)
=
-\frac12
\sum_{x,y\in V}\omega_{xy}\nabla_{xy}f\,\nabla_{xy}g
=
\sum_{x\in V}f(x)\Delta g(x)\mu(x).
\end{equation}

The above definitions and the integration by parts formula are standard in
analysis on weighted graphs. For further background, we refer to
\cite{Grigoryan2018}.

We shall also use the following maximum principle for superharmonic functions
on connected, locally finite graphs; see \cite[Lemma 4.1]{MonticelliPunzoSomaglia2025}.

\begin{lemma}\label{L2.1}
Let \((V,\omega,\mu)\) be a connected, locally finite weighted graph. If
\(u:V\to[0,\infty)\) satisfies
\[
\Delta u\le 0
\qquad\text{in } V,
\]
then either \(u\equiv 0\) in \(V\), or
\[
u(x)>0,\qquad x\in V.
\]
\end{lemma}

\subsection{The transformed graph Laplacian}\label{subs-Delta-H}

Let \((V,\omega,\mu)\) be a connected, locally finite weighted graph, and let
\[
H:V\to(0,\infty).
\]
We define the weighted graph
\[
(V,\omega^H,\mu_H)
\]
by
\[
\omega^H_{xy}=H(x)H(y)\omega_{xy},
\qquad x,y\in V,
\]
and
\[
\mu_H(x)=H^2(x)\mu(x),
\qquad x\in V.
\]

The graph Laplacian associated with \((V,\omega^H,\mu_H)\), called the
\(H\)-Laplacian, is denoted by \(\Delta_H\). Namely, for a function \(f:V\to\mathbb R\),
\[
\Delta_H f(x)
=
\frac{1}{H^2(x)\mu(x)}
\sum_{y\sim x}
\omega_{xy}H(x)H(y)\bigl(f(y)-f(x)\bigr),
\qquad x\in V.
\]

Notice that, if \(H\equiv 1\), then \((V,\omega^H,\mu_H)=(V,\omega,\mu)\) and
\(\Delta_H=\Delta\).

\begin{remark}\label{rem:max-principle-H}
Since \(H>0\), the transformed graph \((V,\omega^H,\mu_H)\) has the same
adjacency relation as \((V,\omega,\mu)\). Hence, if \((V,\omega,\mu)\) is
connected and locally finite, then so is \((V,\omega^H,\mu_H)\). Consequently,
the maximum principle stated in Lemma~\ref{L2.1} applies to the operator
\(\Delta_H\). In particular, if \(z:V\to[0,\infty)\) satisfies
\[
\Delta_H z\le 0
\qquad\text{in } V,
\]
then either \(z\equiv 0\) in \(V\), or
\[
z(x)>0,\qquad x\in V.
\]
\end{remark}

\section{Main result}\label{sec3}

Here and in the sequel, let \((V,\omega,\mu)\) be a connected, locally finite
weighted graph.  For a function \(f:V\to\mathbb R\), we denote by \(f^+\) its
positive part, that is,
\[
f^+(x)=\max\{f(x),0\},
\qquad x\in V.
\]

We study problem \eqref{P} under the following assumptions:

\begin{itemize}
    \item[{\rm(A1)}] there exists a positive function \(H:V\to(0,\infty)\) such that
    \[
    \Delta H+w(x)H=0
    \qquad\text{in } V;
    \]

    \item[{\rm(A2)}] there exist constants \(\gamma>0\), \(R_0>0\), and \(C>0\),
    a family of finitely supported functions
    \[
    \varphi_R:V\to[0,1],
    \qquad R\ge R_0,
    \]
    and a family of subsets \(E_R\subset V\) such that
    \begin{equation}\label{limvarphi}
    \lim_{R\to\infty}\varphi_R(x)=1
    \qquad\text{for every } x\in V,
    \end{equation}
    and
    \begin{equation}\label{Deltavarphi}
    (-\Delta_H\varphi_R)^+(x)
    \le
    \frac{C}{R^\gamma}\mathbf 1_{E_R}(x),
    \qquad x\in V,\ R\ge R_0;
    \end{equation}

    \item[{\rm(A3)}] the sets \(E_R\) escape to infinity, in the sense that, for
    every finite set \(K\subset V\), there exists \(R_K>0\) such that
    \[
    E_R\cap K=\varnothing
    \qquad\text{for all } R\ge R_K.
    \]
\end{itemize}

\begin{remark}
Assumption {\rm(A2)} is a capacitary condition on the graph expressed directly
in terms of cut-off functions. More precisely, the sets \(E_R\) represent the
region where the positive part of the \(H\)-Laplacian of \(\varphi_R\) may be
nonzero, and \eqref{Deltavarphi} gives a quantitative control of this positive
part.

This formulation is different from the metric one used in
\cite{MonticelliPunzoSomaglia2025}. There, the cut-off functions are built from
a pseudo-metric \(d\), typically in the form
\[
\psi\left(\frac{d(x,x_0)}{R}\right),
\]
and the corresponding estimates are controlled by weighted volumes of
pseudo-metric annuli. In contrast, assumption {\rm(A2)} does not prescribe a
pseudo-metric structure. The relevant sets \(E_R\) are chosen according to the
actual support of the cut-off estimate, rather than being imposed a priori by
metric annuli. Thus the present formulation separates the cut-off estimate from any prescribed
metric geometry and can be applied to graphs for which annular volume conditions
may not capture the relevant capacitary behavior.
\end{remark}

\begin{remark}
Assumption {\rm(A1)} allows us to incorporate the potential term \(w u\) into
a transformed graph structure. Indeed, if \(u=Hz\), then the identity
\[
\Delta u+w u=H\Delta_H z
\]
holds, where \(\Delta_H\) is the weighted Laplacian associated with the
transformed weights
\[
\omega^H_{xy}=H(x)H(y)\omega_{xy},
\qquad
\mu_H(x)=H^2(x)\mu(x).
\]
Thus the potential \(w\) is not treated as a perturbative term. Instead, it is
absorbed into the geometry of the transformed graph. This is the reason why the
cut-off condition in {\rm(A2)} is formulated in terms of \(\Delta_H\) rather
than the original Laplacian \(\Delta\).
\end{remark}

We now state the main nonexistence result of this paper.

\begin{theorem}\label{T3.1}
Assume that {\rm(A1)}--{\rm(A3)} hold. Suppose moreover that
\begin{equation}\label{capacity-condition}
\sum_{x\in E_R}
\mu(x)H(x)v^{-\frac{1}{\sigma-1}}(x)
=
O\left(R^{\frac{\gamma\sigma}{\sigma-1}}\right)
\qquad\text{as } R\to\infty.
\end{equation}
Then the only nonnegative solution of \eqref{P} is \(u\equiv 0\) in \(V\).
\end{theorem}

\begin{remark}
\begin{itemize}
    \item[{\rm(i)}] In Section~\ref{sec5}, we present a comparison example with
    \(w\equiv0\). This example shows that Theorem~\ref{T3.1} applies in a
    situation where neither the main result of
\cite{GuHuangSun2023} nor that of \cite{MonticelliPunzoSomaglia2025} applies.

   \item[{\rm(ii)}] In Section~\ref{sec6}, we discuss the sharpness of the
    capacitary condition \eqref{capacity-condition}. More precisely, we show
    that the growth exponent appearing in \eqref{capacity-condition} cannot, in
    general, be improved.
\end{itemize}
\end{remark}

\section{Proof of Theorem~\ref{T3.1}}\label{sec4}
\begin{proof}
Assume that {\rm(A1)}--{\rm(A3)} hold, as well as \eqref{capacity-condition}.
We argue by contradiction and suppose that \(u:V\to[0,\infty)\) is a nontrivial
solution of \eqref{P}.

Define
\[
z(x)=\frac{u(x)}{H(x)},
\qquad x\in V.
\]
Since \(H>0\) in \(V\) and \(u\) is nontrivial, we have \(z\ge 0\) in \(V\)
and \(z\not\equiv 0\) in \(V\).

We first show that \(z\) satisfies
\begin{equation}\label{ineq-z}
\Delta_H z+v(x)H^{\sigma-1}(x)z^\sigma\le 0
\qquad\text{in } V,
\end{equation}
where \(\Delta_H\) is the graph Laplacian associated with
\((V,\omega^H,\mu_H)\), with
\[
\omega^H_{xy}=H(x)H(y)\omega_{xy},
\qquad
\mu_H(x)=H^2(x)\mu(x).
\]
Indeed, for every \(x\in V\), we have
\[
\begin{aligned}
\Delta_H z(x)
&=
\frac{1}{H^2(x)\mu(x)}
\sum_{y\sim x}
\omega_{xy}H(x)H(y)
\left(\frac{u(y)}{H(y)}-\frac{u(x)}{H(x)}\right)\\
&=
\frac{1}{H^2(x)\mu(x)}
\sum_{y\sim x}
\omega_{xy}\bigl(H(x)u(y)-H(y)u(x)\bigr)\\
&=
\frac{1}{H^2(x)\mu(x)}
\sum_{y\sim x}
\omega_{xy}
\Bigl(H(x)\bigl(u(y)-u(x)\bigr)
-u(x)\bigl(H(y)-H(x)\bigr)\Bigr)\\
&=
\frac{1}{H(x)}
\frac{1}{\mu(x)}
\sum_{y\sim x}\omega_{xy}\bigl(u(y)-u(x)\bigr)
-
\frac{u(x)}{H^2(x)}
\frac{1}{\mu(x)}
\sum_{y\sim x}\omega_{xy}\bigl(H(y)-H(x)\bigr)\\
&=
\frac{\Delta u(x)}{H(x)}
-
\frac{u(x)\Delta H(x)}{H^2(x)}.
\end{aligned}
\]
Using {\rm(A1)}, namely \(\Delta H(x)=-w(x)H(x)\), we obtain
\[
\Delta_H z(x)
=
\frac{\Delta u(x)+w(x)u(x)}{H(x)}.
\]
Since \(u\) satisfies \eqref{P}, it follows that
\[
\Delta_H z(x)
\le
-\frac{v(x)u^\sigma(x)}{H(x)}
=
-v(x)H^{\sigma-1}(x)z^\sigma(x),
\]
which proves \eqref{ineq-z}.

Set
\[
v_H(x)=v(x)H^{\sigma-1}(x),
\qquad x\in V.
\]
Then \eqref{ineq-z} becomes
\begin{equation}\label{ineq-zH}
\Delta_H z+v_H(x)z^\sigma\le 0
\qquad\text{in } V.
\end{equation}
Since \(v_H>0\) in \(V\), we have
\[
\Delta_H z\le 0
\qquad\text{in } V.
\]
By Remark~\ref{rem:max-principle-H}, since \(z\ge 0\) in \(V\)
and \(z\not\equiv 0\) in \(V\),  we have
\begin{equation}\label{z-positive}
z(x)>0,
\qquad x\in V.
\end{equation}

Let
\[
s>\frac{\sigma}{\sigma-1}.
\]
Multiplying \eqref{ineq-zH} by \(\varphi_R^s(x)\mu_H(x)\) and summing over
\(V\), we obtain
\begin{equation}\label{IR1}
\sum_{x\in V}\mu_H(x)v_H(x)z^\sigma(x)\varphi_R^s(x)
\le
-\sum_{x\in V}\mu_H(x)\Delta_H z(x)\varphi_R^s(x).
\end{equation}
Since \(\varphi_R\) has finite support, the integration by parts formula
\eqref{Int-parts}, applied to the weighted graph \((V,\omega^H,\mu_H)\), gives
\[
\sum_{x\in V}\mu_H(x)\Delta_H z(x)\varphi_R^s(x)
=
\sum_{x\in V}\mu_H(x)z(x)\Delta_H(\varphi_R^s)(x).
\]
Therefore, by \eqref{IR1}, we have
\begin{equation}\label{IR2}
\sum_{x\in V}\mu_H(x)v_H(x)z^\sigma(x)\varphi_R^s(x)
\le
-\sum_{x\in V}\mu_H(x)z(x)\Delta_H(\varphi_R^s)(x).
\end{equation}

Here and in the sequel, \(C\) denotes a positive constant independent of \(R\),
whose value may change from line to line.

Since \(s>1\), the function \(t\mapsto t^s\) is convex on \([0,\infty)\). Hence,
for all \(a,b\ge 0\),
\[
a^s-b^s\le s a^{s-1}(a-b).
\]
Applying this inequality with
\[
a=\varphi_R(x),
\qquad
b=\varphi_R(y),
\]
we obtain, for every \(x,y\in V\),
\[
\varphi_R^s(x)-\varphi_R^s(y)
\le
s\varphi_R^{s-1}(x)\bigl(\varphi_R(x)-\varphi_R(y)\bigr).
\]
Therefore, for every \(x\in V\),
\[
\begin{aligned}
-\Delta_H(\varphi_R^s)(x)
&=
\frac{1}{H^2(x)\mu(x)}
\sum_{y\sim x}
\omega_{xy}H(x)H(y)
\bigl(\varphi_R^s(x)-\varphi_R^s(y)\bigr)\\
&\le
s\varphi_R^{s-1}(x)
\frac{1}{H^2(x)\mu(x)}
\sum_{y\sim x}
\omega_{xy}H(x)H(y)
\bigl(\varphi_R(x)-\varphi_R(y)\bigr)\\
&=
s\varphi_R^{s-1}(x)(-\Delta_H\varphi_R)(x)\\
&\le
s\varphi_R^{s-1}(x)(-\Delta_H\varphi_R)^+(x).
\end{aligned}
\]
By \eqref{Deltavarphi}, it follows that
\[
-\Delta_H(\varphi_R^s)(x)
\le
\frac{C}{R^\gamma}\varphi_R^{s-1}(x)\mathbf 1_{E_R}(x),
\qquad x\in V,\ R\ge R_0.
\]
Combining this estimate with \eqref{IR2}, we obtain
\[
\sum_{x\in V}\mu_H(x)v_H(x)z^\sigma(x)\varphi_R^s(x)
\le
\frac{C}{R^\gamma}
\sum_{x\in E_R}\mu_H(x)z(x)\varphi_R^{s-1}(x).
\]
Set
\[
I_R=
\sum_{x\in V}\mu_H(x)v_H(x)z^\sigma(x)\varphi_R^s(x).
\]
Then
\begin{equation}\label{IR3}
I_R
\le
\frac{C}{R^\gamma}
\sum_{x\in E_R}\mu_H(x)z(x)\varphi_R^{s-1}(x).
\end{equation}

By H\"older's inequality,
\[
\begin{aligned}
\sum_{x\in E_R}\mu_H(x)z(x)\varphi_R^{s-1}(x)
&\le
\left(
\sum_{x\in E_R}
\mu_H(x)v_H(x)z^\sigma(x)\varphi_R^s(x)
\right)^{\frac{1}{\sigma}} \\
&\quad\times
\left(
\sum_{x\in E_R}
\mu_H(x)v_H^{-\frac{1}{\sigma-1}}(x)
\varphi_R^{s-\frac{\sigma}{\sigma-1}}(x)
\right)^{\frac{\sigma-1}{\sigma}}.
\end{aligned}
\]
Since \(0\le \varphi_R\le 1\) and \(s>\frac{\sigma}{\sigma-1}\), we have
\[
\varphi_R^{s-\frac{\sigma}{\sigma-1}}(x)\le 1,
\qquad x\in V.
\]
Hence
\[
\sum_{x\in E_R}\mu_H(x)z(x)\varphi_R^{s-1}(x)
\le
\left(
\sum_{x\in E_R}
\mu_H(x)v_H(x)z^\sigma(x)\varphi_R^s(x)
\right)^{\frac{1}{\sigma}}
\left(
\sum_{x\in E_R}
\mu_H(x)v_H^{-\frac{1}{\sigma-1}}(x)
\right)^{\frac{\sigma-1}{\sigma}}.
\]
Using this estimate in \eqref{IR3}, we get
\begin{equation}\label{IR4}
I_R
\le
\frac{C}{R^\gamma}
\left(
\sum_{x\in E_R}
\mu_H(x)v_H(x)z^\sigma(x)\varphi_R^s(x)
\right)^{\frac{1}{\sigma}}
\left(
\sum_{x\in E_R}
\mu_H(x)v_H^{-\frac{1}{\sigma-1}}(x)
\right)^{\frac{\sigma-1}{\sigma}}.
\end{equation}

Since
\[
\mu_H(x)v_H^{-\frac{1}{\sigma-1}}(x)
=
\mu(x)H(x)v^{-\frac{1}{\sigma-1}}(x),
\]
assumption \eqref{capacity-condition} implies
\[
\sum_{x\in E_R}
\mu_H(x)v_H^{-\frac{1}{\sigma-1}}(x)
=
O\left(R^{\frac{\gamma\sigma}{\sigma-1}}\right)
\qquad\text{as } R\to\infty.
\]
Therefore, it follows from \eqref{IR4} that, for \(R\) sufficiently large,
\begin{equation}\label{IR5}
I_R
\le
C
\left(
\sum_{x\in E_R}
\mu_H(x)v_H(x)z^\sigma(x)\varphi_R^s(x)
\right)^{\frac{1}{\sigma}}.
\end{equation}
Since
\[
\sum_{x\in E_R}
\mu_H(x)v_H(x)z^\sigma(x)\varphi_R^s(x)
\le I_R,
\]
we infer from \eqref{IR5} that
\[
I_R\le C I_R^{\frac{1}{\sigma}}.
\]
Thus,
\begin{equation}\label{IR-bound}
I_R\le C
\end{equation}
for all sufficiently large \(R\).

By Fatou's lemma, together with \eqref{limvarphi} and \eqref{IR-bound}, we get
\[
\sum_{x\in V}\mu_H(x)v_H(x)z^\sigma(x)
\le
\liminf_{R\to\infty} I_R
<\infty.
\]
Hence
\begin{equation}\label{global-integrability}
\sum_{x\in V}\mu_H(x)v_H(x)z^\sigma(x)<\infty.
\end{equation}

We claim that
\begin{equation}\label{tail-ER}
\sum_{x\in E_R}\mu_H(x)v_H(x)z^\sigma(x)\to 0
\qquad\text{as } R\to\infty.
\end{equation}
Indeed, let \(\varepsilon>0\). By \eqref{global-integrability}, there exists a
finite set \(K_\varepsilon\subset V\) such that
\[
\sum_{x\in V\setminus K_\varepsilon}
\mu_H(x)v_H(x)z^\sigma(x)<\varepsilon.
\]
By assumption {\rm(A3)}, there exists \(R_\varepsilon>0\) such that
\[
E_R\cap K_\varepsilon=\varnothing
\qquad\text{for all } R\ge R_\varepsilon.
\]
Thus, for \(R\ge R_\varepsilon\),
\[
E_R\subset V\setminus K_\varepsilon,
\]
and consequently
\[
\sum_{x\in E_R}\mu_H(x)v_H(x)z^\sigma(x)
\le
\sum_{x\in V\setminus K_\varepsilon}
\mu_H(x)v_H(x)z^\sigma(x)
<\varepsilon.
\]
This proves \eqref{tail-ER}.

Since \(0\le \varphi_R\le 1\), \eqref{tail-ER} gives
\[
\sum_{x\in E_R}
\mu_H(x)v_H(x)z^\sigma(x)\varphi_R^s(x)
\to 0
\qquad\text{as } R\to\infty.
\]
Hence, by \eqref{IR5},
\[
I_R\to 0
\qquad\text{as } R\to\infty.
\]
Using Fatou's lemma again, together with \eqref{limvarphi}, we obtain
\[
\sum_{x\in V}\mu_H(x)v_H(x)z^\sigma(x)
\le
\liminf_{R\to\infty}I_R=0.
\]
Therefore,
\[
\sum_{x\in V}\mu_H(x)v_H(x)z^\sigma(x)=0.
\]
Since \(\mu_H>0\), \(v_H>0\), and \(z\ge 0\) in \(V\), we obtain  \(z\equiv 0\) in \(V\), which contradicts \eqref{z-positive}.
This contradiction completes the proof.
\end{proof}

\section{A comparison example with \(w\equiv0\)}\label{sec5}

In this section, we take \(w\equiv 0\) and \(v\equiv 1\) in \eqref{P} and
construct a weighted graph for which neither the main result of
\cite{GuHuangSun2023} nor that of \cite{MonticelliPunzoSomaglia2025} applies,
whereas Theorem~\ref{T3.1} still implies that every nonnegative solution is
identically zero.

\begin{example}\label{ex1}

Fix
\[
D>2.
\]
Let
\[
A=\{a_0,a_1,a_2,\ldots\}
\]
be a countably infinite set of vertices. For each \(n\ge 0\), set
\[
B_n=\{b_n\}.
\]
Assume that the sets
\[
A,\ B_0,\ B_1,\ldots
\]
are pairwise disjoint, and define
\[
V=A\cup\bigcup_{n\ge 0}B_n.
\]

Set
\[
c_n=(n+1)^{D-1},
\qquad n\ge 0.
\]
We define 
\[
\omega:V\times V\to[0,\infty)
\]
by
\[
\omega_{a_n a_{n+1}}=\omega_{a_{n+1}a_n}=c_n,
\qquad
\omega_{a_n b_n}=\omega_{b_n a_n}=1,
\qquad n\ge 0,
\]
and
\[
\omega_{xx}=0,
\qquad x\in V,
\]
with \(\omega_{xy}=0\) for all remaining pairs \((x,y)\in V\times V\).

We define the vertex measure by
\[
\mu(a_n)=c_n+c_{n-1},
\qquad n\ge 0,
\]
with the convention
\[
c_{-1}=0,
\]
and by
\[
\mu(b_n)=e^{n^2},
\qquad n\ge 0.
\]
Then \(\mu(x)>0\) for every \(x\in V\).

By construction, each vertex has finitely many neighbors. Indeed,
\[
\{y\in V:\ y\sim a_0\}=\{a_1,b_0\},
\]
\[
\{y\in V:\ y\sim a_n\}=\{a_{n-1},a_{n+1},b_n\},
\qquad n\ge 1,
\]
and
\[
\{y\in V:\ y\sim b_n\}=\{a_n\},
\qquad n\ge 0.
\]
Hence \((V,\omega,\mu)\) is locally finite. It is also connected, since every
\(b_n\) is connected to \(a_n\), and the vertices \(\{a_n\}_{n\ge 0}\) form a
connected backbone.

We consider problem \eqref{P} with
\[
v\equiv 1,
\qquad
w\equiv 0.
\]
Namely, we consider the elliptic inequality
\begin{equation}\label{P-ex1}
\Delta u+u^\sigma\le 0
\qquad\text{in } V,
\end{equation}
where \(\sigma>1\).

We now verify that {\rm(A1)}--{\rm(A3)} are satisfied.

\(\bullet\)  Since \(w\equiv 0\),  {\rm(A1)} is satisfied with
    \[
    H\equiv 1.
    \]
    Indeed,
    \[
    \Delta H+w(x)H=0
    \qquad\text{in } V.
    \]
   
\(\bullet\)  Let \(\eta\in C^2([0,\infty))\) be such that
    \[
    0\le \eta\le 1,\qquad
    \eta(t)=1\quad\text{for } 0\le t\le 1,
    \qquad
    \eta(t)=0\quad\text{for } t\ge 2.
    \]
    For \(R\ge 2\), define
    \[
    \varphi_R(a_n)=\varphi_R(b_n)=\eta\left(\frac{n}{R}\right),
    \qquad n\ge 0.
    \]
    Then
    \[
    0\le \varphi_R\le 1
    \qquad\text{in } V.
    \]
    Moreover, since \(\eta(t)=0\) for \(t\ge 2\), we have
    \[
    \operatorname{supp}\varphi_R
    \subset
    \{a_n:\ 0\le n<2R\}\cup\{b_n:\ 0\le n<2R\}.
    \]
    Therefore, \(\varphi_R\) has finite support.
    
        Also, for every \(x\in V\), there exists \(n_x\ge 0\) such that
    \[
    x\in\{a_{n_x},b_{n_x}\}.
    \]
    Hence, for all \(R\ge \max\{2,n_x\}\), we have
    \[
    \varphi_R(x)=\eta\left(\frac{n_x}{R}\right)=1.
    \]
    Therefore,
    \[
    \lim_{R\to\infty}\varphi_R(x)=1,
    \qquad x\in V,
    \]
    and \eqref{limvarphi} is satisfied.

     For \(R\ge 2\), set
    \[
    E_R=\{a_n:\ R-1\le n\le 2R+1\}.
    \]
    We claim that
    \begin{equation}\label{Deltavarphi-ex1}
    (-\Delta\varphi_R)^+(x)
    \le
    \frac{C}{R^2}\mathbf 1_{E_R}(x),
    \qquad x\in V,\ R\ge 2.
    \end{equation}
    We distinguish two cases.
    
   \noindent  \textbf{Case 1.} \(x=b_n\) for some \(n\ge 0\).    Since \(b_n\) is adjacent only to \(a_n\), and since
    \[
    \varphi_R(b_n)=\varphi_R(a_n),
    \]
    we have
    \[
    \Delta\varphi_R(b_n)
    =
    \frac{1}{\mu(b_n)}
    \omega_{b_n a_n}\bigl(\varphi_R(a_n)-\varphi_R(b_n)\bigr)
    =0.
    \]
    Thus,
    \[
    (-\Delta\varphi_R)^+(b_n)=0.
    \]
    
 \noindent \textbf{Case 2.} Let \(x=a_n\) for some \(n\ge 0\). The edge joining
    \(a_n\) to \(b_n\) does not contribute to \(\Delta\varphi_R(a_n)\), because
    \[
    \varphi_R(b_n)=\varphi_R(a_n).
    \]
    Therefore, by the definition of \(\mu\), for \(n\ge 1\),
    \begin{equation}\label{Deltaan-ex1}
    \begin{aligned}
    \Delta\varphi_R(a_n)
    &=
    \frac{1}{c_n+c_{n-1}}
    \Bigg[
    c_n\left(\eta\left(\frac{n+1}{R}\right)-\eta\left(\frac{n}{R}\right)\right)\\
    &\qquad\qquad\qquad
    +
    c_{n-1}\left(\eta\left(\frac{n-1}{R}\right)-\eta\left(\frac{n}{R}\right)\right)
    \Bigg].
    \end{aligned}
    \end{equation}
    For \(n=0\), the same formula holds with the term involving \(c_{n-1}\)
    omitted.

     We distinguish two subcases.

 \noindent \textbf{Subcase 2.1.} Assume that \(a_n\notin E_R\). By the definition of
    \(E_R\), either \(n<R-1\) or \(n>2R+1\).

    In the first case, for \(R\ge 2\), we have
    \[
    \frac{n-1}{R},\ \frac{n}{R},\ \frac{n+1}{R}\le 1,
    \]
    whenever the terms are defined, and hence all the corresponding values of
    \(\eta\) are equal to \(1\). In the second case, we have
    \[
    \frac{n-1}{R},\ \frac{n}{R},\ \frac{n+1}{R}\ge 2,
    \]
    and hence all the corresponding values of \(\eta\) are equal to \(0\).
    Therefore, by \eqref{Deltaan-ex1},
    \[
    \Delta\varphi_R(a_n)=0.
    \]
    Thus,
    \[
    (-\Delta\varphi_R)^+(a_n)=0.
    \]
   
\noindent \textbf{Subcase 2.2.} Assume that \(a_n\in E_R\). Since \(R\ge 2\), this
    implies \(n\ge 1\). Set
    \[
    \eta_n=\eta\left(\frac{n}{R}\right),
    \qquad
    \alpha_n=\frac{c_n}{c_n+c_{n-1}},
    \qquad
    \beta_n=\frac{c_{n-1}}{c_n+c_{n-1}}.
    \]
    Then, by \eqref{Deltaan-ex1},
    \[
    \Delta\varphi_R(a_n)
    =
    \alpha_n(\eta_{n+1}-\eta_n)
    +
    \beta_n(\eta_{n-1}-\eta_n).
    \]
    Equivalently,
    \begin{equation}\label{newf-Delta}
    \Delta\varphi_R(a_n)
    =
    \alpha_n(\eta_{n+1}-2\eta_n+\eta_{n-1})
    +
    (\alpha_n-\beta_n)(\eta_n-\eta_{n-1}).
    \end{equation}
    Since \(\eta\in C^2([0,\infty))\) and \(\eta\) is compactly supported, there
    exists \(C>0\) such that, for all \(t\ge 0\),
    \[
    |\eta'(t)|\le C,
    \qquad
    |\eta''(t)|\le C.
    \]
    By the mean value theorem, for \(n\ge 1\),
    \begin{equation}\label{esteta1}
    |\eta_n-\eta_{n-1}|
    =
    \left|\eta\left(\frac{n}{R}\right)-\eta\left(\frac{n-1}{R}\right)\right|
    \le \frac{C}{R}.
    \end{equation}
    Moreover, Taylor's formula gives
    \[
    \eta\left(\frac{n+1}{R}\right)
    =
    \eta\left(\frac{n}{R}\right)
    +
    \frac{1}{R}\eta'\left(\frac{n}{R}\right)
    +
    \frac{1}{2R^2}\eta''(\xi_{n,+}),
    \]
    for some \(\xi_{n,+}\) between \(\frac{n}{R}\) and \(\frac{n+1}{R}\), and
    \[
    \eta\left(\frac{n-1}{R}\right)
    =
    \eta\left(\frac{n}{R}\right)
    -
    \frac{1}{R}\eta'\left(\frac{n}{R}\right)
    +
    \frac{1}{2R^2}\eta''(\xi_{n,-}),
    \]
    for some \(\xi_{n,-}\) between \(\frac{n-1}{R}\) and \(\frac{n}{R}\).
    Adding these two identities, we obtain
    \[
    \eta_{n+1}-2\eta_n+\eta_{n-1}
    =
    \frac{1}{2R^2}
    \left(\eta''(\xi_{n,+})+\eta''(\xi_{n,-})\right).
    \]
    Since \(\eta''\) is bounded on \([0,\infty)\), it follows that
    \begin{equation}\label{esteta2}
    |\eta_{n+1}-2\eta_n+\eta_{n-1}|
    \le
    \frac{C}{R^2}.
    \end{equation}
    Moreover,
    \[
    n\frac{c_n-c_{n-1}}{c_n+c_{n-1}}
    \to \frac{D-1}{2}
    \qquad\text{as } n\to\infty.
    \]
    Hence, there exists \(C>0\) such that
    \begin{equation}\label{esteta3}
    |\alpha_n-\beta_n|
    =
    \frac{|c_n-c_{n-1}|}{c_n+c_{n-1}}
    \le \frac{C}{n},
    \qquad n\ge 1.
    \end{equation}
    We also have
    \begin{equation}\label{esteta4}
    0\le \alpha_n\le 1.
    \end{equation}
    Combining \eqref{newf-Delta}, \eqref{esteta1}, \eqref{esteta2},
    \eqref{esteta3}, and \eqref{esteta4}, we obtain
    \[
    |\Delta\varphi_R(a_n)|
    \le
    \frac{C}{R^2}+\frac{C}{nR},
    \qquad n\ge 1.
    \]
    On the other hand, by the definition of \(E_R\), we have
    \[
    R-1\le n\le 2R+1.
    \]
    Since \(R\ge 2\), this implies
    \[
    n\ge R-1\ge \frac{R}{2}.
    \]
    Therefore,
    \[
    |\Delta\varphi_R(a_n)|
    \le
    \frac{C}{R^2}.
    \]
    Consequently,
    \[
    (-\Delta\varphi_R)^+(a_n)
    \le
    \frac{C}{R^2}\mathbf 1_{E_R}(a_n).
    \]
   
    Combining Cases 1 and 2, we obtain \eqref{Deltavarphi-ex1}. Since
    \(H\equiv 1\), we have \(\Delta_H=\Delta\). Thus, \eqref{Deltavarphi} is
    satisfied with \(\gamma=2\), and therefore  {\rm(A2)} is satisfied.

 \(\bullet\)      We finally verify  {\rm(A3)}. Let \(K\subset V\) be a finite set.
    Then there exists \(N_K\ge 0\) such that
    \[
    K\subset
    \{a_0,\ldots,a_{N_K}\}\cup\{b_0,\ldots,b_{N_K}\}.
    \]
    If \(R>N_K+1\), then
    \[
    E_R\cap K=\varnothing.
    \]
    Thus, the sets \(E_R\) escape to infinity, and  {\rm(A3)} is
    satisfied.

It remains to verify condition \eqref{capacity-condition}. Since
\(H\equiv 1\) and \(v\equiv 1\), we have
\[
\sum_{x\in E_R}
\mu(x)H(x)v^{-\frac{1}{\sigma-1}}(x)
=
\sum_{x\in E_R}\mu(x).
\]
By the definition of \(E_R\), this gives
\[
\sum_{x\in E_R}\mu(x)
=
\sum_{R-1\le n\le 2R+1}\mu(a_n).
\]
Using
\[
\mu(a_n)=c_n+c_{n-1}
\]
and
\[
c_n=(n+1)^{D-1},
\]
we obtain
\[
\sum_{x\in E_R}\mu(x)
\le
C\sum_{R-1\le n\le 2R+1} n^{D-1}
\le
C R^D.
\]
Therefore,
\[
\sum_{x\in E_R}
\mu(x)H(x)v^{-\frac{1}{\sigma-1}}(x)
=
O(R^D)
\qquad\text{as } R\to\infty.
\]
Since \(\gamma=2\), condition \eqref{capacity-condition} is satisfied whenever
\[
D\le \frac{2\sigma}{\sigma-1}.
\]
Consequently, by Theorem~\ref{T3.1}, if
\[
1<\sigma\le \frac{D}{D-2},
\]
then the only nonnegative solution of \eqref{P-ex1} is \(u\equiv 0\) in \(V\).

We now show that the main structural assumption in \cite{GuHuangSun2023},
namely condition \eqref{p0-cd}, is not satisfied in the present example.
Indeed, for the adjacent vertices \(b_n\) and \(a_n\), we have
\[
\omega_{b_n a_n}=1,
\qquad
\mu(b_n)=e^{n^2}.
\]
Therefore,
\[
\frac{\omega_{b_n a_n}}{\mu(b_n)}
=
e^{-n^2}\to 0
\qquad\text{as } n\to\infty.
\]
Hence, there is no \(p_0>1\) such that \eqref{p0-cd} holds for all adjacent
vertices \(x,y\in V\). Consequently, the main result of
\cite{GuHuangSun2023} cannot be applied to this example.

We next show that the main result of \cite{MonticelliPunzoSomaglia2025} (Theorem~\ref{T-Punzo}) cannot be applied either. In the present example, \(v\equiv 1\). Hence the weighted
annular volume condition in {\rm(H4)} reduces to
\begin{equation}\label{cd-Punzo}
\mu\bigl(B_{2R}(x_0)\setminus B_R(x_0)\bigr)
=
O\left(R^{\frac{(1+\alpha)\sigma}{\sigma-1}}\right)
\qquad\text{as } R\to\infty,
\end{equation}
where \(B_R(x_0)\) denotes the ball centered at \(x_0\) with radius \(R\), taken
with respect to the pseudo-metric \(d\). Here, \(d\), \(x_0\), and
\(\alpha\in[0,1]\) are those appearing in {\rm(H2)}--{\rm(H4)}. Let
\[
q=\frac{(1+\alpha)\sigma}{\sigma-1}.
\]
By \eqref{cd-Punzo}, there exist \(C>0\) and \(R_0>0\) such that
\[
\mu\bigl(B_{2R}(x_0)\setminus B_R(x_0)\bigr)\le C R^q,
\qquad R\ge R_0.
\]
Choose \(k_0\in\mathbb N\) such that
\[
2^{k_0}\ge R_0.
\]
Then, for every \(k\ge k_0\),
\[
\mu\bigl(B_{2^{k+1}}(x_0)\setminus B_{2^k}(x_0)\bigr)
\le
C2^{kq}.
\]
 For \(R\ge 2^{k_0}\), choose \(N\in\mathbb N\) such that
\[
2^N\le R<2^{N+1}.
\]
Then
\[
B_R(x_0)\subset B_{2^{N+1}}(x_0),
\]
and hence
\[
\begin{aligned}
\mu(B_R(x_0))
&\le \mu(B_{2^{N+1}}(x_0))\\
&\le
\mu(B_{2^{k_0}}(x_0))
+
\sum_{k=k_0}^{N}
\mu\bigl(B_{2^{k+1}}(x_0)\setminus B_{2^k}(x_0)\bigr)\\
&\le
\mu(B_{2^{k_0}}(x_0))
+
C\sum_{k=k_0}^{N}2^{kq}\\
&\le
C R^q.
\end{aligned}
\]
 In the last inequality, we used \(q>0\), \(2^N\le R\), and absorbed the fixed
quantity \(\mu(B_{2^{k_0}}(x_0))\) into the constant \(C\).  Thus,
\begin{equation}\label{ball-growth-Punzo}
\mu(B_R(x_0))\le C R^q
\qquad\text{for all sufficiently large } R.
\end{equation}

On the other hand, by the finite jump-size condition in {\rm(H2)},
\[
d(x,y)\le j
\qquad\text{whenever } x\sim y.
\]
Since \(x_0\in V\), there exists \(m\ge 0\) such that
\[
x_0\in\{a_m,b_m\}.
\]
For every \(n\ge m\), the vertices \(x_0\) and \(b_n\) can be joined by a path
with at most \(n-m+2\) edges. Indeed, if \(x_0=a_m\), one may use the path
\[
a_m,a_{m+1},\ldots,a_n,b_n,
\]
which has \(n-m+1\) edges. If \(x_0=b_m\), one may use the path
\[
b_m,a_m,a_{m+1},\ldots,a_n,b_n,
\]
which has \(n-m+2\) edges. Therefore, in both cases,
\[
d(x_0,b_n)\le j(n-m+2)\le C(n+1),
\qquad n\ge m.
\]
Consequently,
\[
b_n\in B_{C(n+1)}(x_0),
\qquad n\ge m.
\]
Then, using \eqref{ball-growth-Punzo}, we obtain, for all sufficiently large \(n\),
\[
e^{n^2}=\mu(b_n)
\le
\mu(B_{C(n+1)}(x_0))
\le
C(n+1)^q,
\]
which is impossible as \(n\to\infty\).

This contradiction shows that the annular volume condition \eqref{cd-Punzo}
cannot be satisfied for any pseudo-metric \(d\) satisfying the finite jump-size
condition in {\rm(H2)}. Therefore, {\rm(H2)} and {\rm(H4)} cannot hold
simultaneously in the present example, and the main result of
\cite{MonticelliPunzoSomaglia2025} cannot be applied.
\end{example}

\section{Sharpness of the capacitary condition}\label{sec6}

In this section, we show that the capacitary condition
\eqref{capacity-condition} in Theorem~\ref{T3.1} is sharp with respect to the
growth exponent of \(R\). More precisely, given \(\varepsilon>0\), we construct
an example with \(v\equiv1\) and a nontrivial potential \(w\) for which the
structural assumptions {\rm(A1)}--{\rm(A3)} are satisfied, but the capacitary
quantity grows like
\[
\sum_{x\in E_R}
\mu(x)H(x)v^{-\frac{1}{\sigma-1}}(x)
\asymp
R^{\frac{\gamma\sigma}{\sigma-1}+\varepsilon},
\]
and \eqref{P} admits a positive nontrivial solution.

\begin{example}

Let
\[
\sigma>1,\qquad \varepsilon>0,\qquad
D=1+\frac{2}{\sigma-1}+\varepsilon,
\]
and set
\[
c_n=(n+1)^{D-1},
\qquad n\ge0.
\]
Let
\[
V=\{a_0,a_1,a_2,\ldots\}
\]
be a countably infinite set of pairwise distinct vertices.

We define 
\[
\omega:V\times V\to[0,\infty)
\] 
by
\[
\omega_{a_n a_{n+1}}=\omega_{a_{n+1}a_n}=c_n,
\qquad n\ge0,
\]
and
\[
\omega_{xx}=0,
\qquad x\in V,
\]
with \(\omega_{xy}=0\) for all remaining pairs \((x,y)\in V\times V\).

We define the vertex measure by
\[
\mu(a_n)=c_n+c_{n-1},
\qquad n\ge0,
\]
where we use the convention
\[
c_{-1}=0.
\]
Then \(\mu(x)>0\) for every \(x\in V\). 

Moreover, each vertex has finitely many neighbors. Indeed,
\[
\{y\in V:\ y\sim a_0\}=\{a_1\},
\]
and
\[
\{y\in V:\ y\sim a_n\}=\{a_{n-1},a_{n+1}\},
\qquad n\ge1.
\]
Hence \((V,\omega,\mu)\) is locally finite. It is also connected by
construction. Indeed, for any \(m,n\ge0\), if \(m<n\), then \(a_m\) and
\(a_n\) are joined by the path
\[
a_m,a_{m+1},\ldots,a_n.
\]
The case \(n<m\) is analogous.

We now define the weight \(v\) and the potential \(w\). We take
\begin{equation}\label{vf-ex2}
v(a_n)=1,
\qquad n\ge0.
\end{equation}
Fix an integer \(n_0\ge1\). We define \(w:V\to\mathbb R\) by
\begin{equation}\label{w-f-ex2}
w(a_n)
=
-\frac{c_n-c_{n-1}}{(c_n+c_{n-1})(n+n_0)},
\qquad n\ge0.
\end{equation}

We consider problem~\eqref{P} on the weighted graph \((V,\omega,\mu)\)
introduced above, with the functions \(v\) and \(w\) defined respectively by
\eqref{vf-ex2} and \eqref{w-f-ex2}. Namely,
\begin{equation}\label{Pb-ex2}
\Delta u(a_n)+w(a_n)u(a_n)+u^\sigma(a_n)\le 0,
\qquad n\ge0.
\end{equation}

We now verify that  {\rm(A1)}--{\rm(A3)} are satisfied.

\(\bullet\) Define \(H:V\to(0,\infty)\) by
\[
H(a_n)=n+n_0,
\qquad n\ge0.
\]
Then, for all \(n\ge1\),
\[
\begin{aligned}
&\Delta H(a_n)+w(a_n)H(a_n)\\
&=
\frac{1}{\mu(a_n)}
\sum_{y\sim a_n}\omega_{a_n y}\bigl(H(y)-H(a_n)\bigr)
-\frac{c_n-c_{n-1}}{c_n+c_{n-1}}\\
&=
\frac{1}{c_n+c_{n-1}}
\left[
\omega_{a_n a_{n-1}}\bigl(H(a_{n-1})-H(a_n)\bigr)
+
\omega_{a_n a_{n+1}}\bigl(H(a_{n+1})-H(a_n)\bigr)
\right]\\
&\quad
-\frac{c_n-c_{n-1}}{c_n+c_{n-1}}\\
&=
\frac{-c_{n-1}+c_n}{c_n+c_{n-1}}
-\frac{c_n-c_{n-1}}{c_n+c_{n-1}}\\
&=0.
\end{aligned}
\]
Similarly, for \(n=0\), we have
\[
\begin{aligned}
\Delta H(a_0)+w(a_0)H(a_0)
&=
\frac{1}{\mu(a_0)}
\sum_{y\sim a_0}\omega_{a_0 y}\bigl(H(y)-H(a_0)\bigr)-1\\
&=
\frac{1}{c_0}\omega_{a_0 a_1}\bigl(H(a_1)-H(a_0)\bigr)-1\\
&=
H(a_1)-H(a_0)-1\\
&=0.
\end{aligned}
\]
Thus,
\[
\Delta H+wH=0
\qquad\text{in }V,
\]
and {\rm(A1)} is satisfied.

\(\bullet\) We next verify {\rm(A2)}. Recall that the transformed graph
associated with \(H\) is given by
\[
\omega^H_{xy}=H(x)H(y)\omega_{xy},
\qquad
\mu_H(x)=H^2(x)\mu(x),
\qquad x,y\in V.
\]
In the present example, this gives
\[
\omega^H_{a_n a_{n+1}}
=
(n+n_0)(n+1+n_0)c_n,
\qquad n\ge0,
\]
and
\[
\mu_H(a_n)
=
(n+n_0)^2(c_n+c_{n-1}),
\qquad n\ge0.
\]

Let \(\eta\in C^2([0,\infty))\) be such that
\[
0\le \eta\le1,\qquad
\eta(t)=1\quad\text{for }0\le t\le1,
\qquad
\eta(t)=0\quad\text{for }t\ge2.
\]
For \(R\ge2\), define
\[
\varphi_R(a_n)=\eta\left(\frac{n}{R}\right),
\qquad n\ge0.
\]
Then \(0\le \varphi_R\le1\), and \(\varphi_R\) has finite support. Moreover,
for every \(n\ge0\),
\[
\lim_{R\to\infty}\varphi_R(a_n)=1.
\]

Set
\[
E_R=\{a_n:\ R-1\le n\le 2R+1\}.
\]
We claim that
\begin{equation}\label{Deltavarphi-ex2}
(-\Delta_H\varphi_R)^+(a_n)
\le
\frac{C}{R^2}\mathbf 1_{E_R}(a_n),
\qquad n\ge0,\ R\ge2.
\end{equation}

First, if \(a_n\notin E_R\), then either \(n<R-1\) or \(n>2R+1\). In both
cases, the three quantities
\[
\varphi_R(a_{n-1}),\qquad \varphi_R(a_n),\qquad \varphi_R(a_{n+1})
\]
are equal whenever they are defined. Hence
\[
\Delta_H\varphi_R(a_n)=0,
\]
and \eqref{Deltavarphi-ex2} follows.

It remains to consider \(a_n\in E_R\). For \(n\ge1\), set
\[
\eta_n=\eta\left(\frac{n}{R}\right),
\]
and define
\[
A_n=
\frac{\omega^H_{a_n a_{n+1}}}{\mu_H(a_n)}
=
\frac{c_n}{c_n+c_{n-1}}\frac{n+1+n_0}{n+n_0},\qquad 
B_n=
\frac{\omega^H_{a_n a_{n-1}}}{\mu_H(a_n)}
=
\frac{c_{n-1}}{c_n+c_{n-1}}\frac{n-1+n_0}{n+n_0}.
\]
Then
\[
\Delta_H\varphi_R(a_n)
=
A_n(\eta_{n+1}-\eta_n)+B_n(\eta_{n-1}-\eta_n).
\]
Equivalently,
\begin{equation}\label{DeltaH-ex2}
\Delta_H\varphi_R(a_n)
=
A_n(\eta_{n+1}-2\eta_n+\eta_{n-1})
+
(A_n-B_n)(\eta_n-\eta_{n-1}).
\end{equation}
As in Example~\ref{ex1}, since \(\eta\in C^2([0,\infty))\), there exists
\(C>0\) such that
\[
|\eta_n-\eta_{n-1}|\le \frac{C}{R},
\qquad
|\eta_{n+1}-2\eta_n+\eta_{n-1}|\le \frac{C}{R^2}.
\]
Moreover,
\[
0\le A_n\le C
\]
and
\[
|A_n-B_n|\le \frac{C}{n},
\qquad n\ge1.
\]
Indeed, the last inequality follows from
\[
n(A_n-B_n)\to \frac{D+1}{2}
\qquad\text{as }n\to\infty.
\]
Therefore, by \eqref{DeltaH-ex2},
\[
|\Delta_H\varphi_R(a_n)|
\le
\frac{C}{R^2}+\frac{C}{nR},
\qquad n\ge1.
\]
Since \(a_n\in E_R\) and \(R\ge2\), we have  \(n\ge R/2\). Hence
\[
|\Delta_H\varphi_R(a_n)|
\le
\frac{C}{R^2},
\qquad a_n\in E_R,
\]
which yields 
\[
(-\Delta_H\varphi_R)^+(a_n)
\le
\frac{C}{R^2},
\qquad a_n\in E_R, \ R\ge2.
\]

Consequently, \eqref{Deltavarphi-ex2} holds, and  {\rm(A2)} is satisfied with
\[
\gamma=2.
\]

\(\bullet\) We finally verify {\rm(A3)}. Let \(K\subset V\) be a finite set.
Since
\[
V=\{a_0,a_1,a_2,\ldots\},
\]
there exists \(N_K\ge0\) such that
\[
K\subset \{a_0,a_1,\ldots,a_{N_K}\}.
\]
If \(R>N_K+1\), then
\[
E_R\cap K=\varnothing.
\]
Thus the sets \(E_R\) escape to infinity, and  {\rm(A3)} is satisfied.

We now compute the capacitary quantity. By the definitions of \(\mu\), \(H\), \(v\), and \(E_R\), we have 
\[ 
\begin{aligned} 
\sum_{x\in E_R} \mu(x)H(x)v^{-\frac{1}{\sigma-1}}(x) 
&= \sum_{R-1\le n\le 2R+1} \mu(a_n)H(a_n)v^{-\frac{1}{\sigma-1}}(a_n)\\ 
&= \sum_{R-1\le n\le 2R+1} (c_n+c_{n-1})(n+n_0). 
\end{aligned} 
\]
Since 
\[ 
c_n=(n+1)^{D-1}, 
\] 
we have 
\[ 
c_n+c_{n-1}\asymp n^{D-1} \qquad\text{as }n\to\infty. 
\] 
Therefore, 
\[ 
(c_n+c_{n-1})(n+n_0)\asymp n^D \qquad\text{as }n\to\infty. 
\]
It follows that 
\[ 
\sum_{x\in E_R} \mu(x)H(x)v^{-\frac{1}{\sigma-1}}(x) \asymp \sum_{R-1\le n\le 2R+1}n^D \asymp R^{D+1}. 
\]
By the choice of \(D\), 
\[ 
D+1 = 2+\frac{2}{\sigma-1}+\varepsilon = \frac{2\sigma}{\sigma-1}+\varepsilon. 
\]
Hence 
\[ 
\sum_{x\in E_R} \mu(x)H(x)v^{-\frac{1}{\sigma-1}}(x) \asymp R^{\frac{2\sigma}{\sigma-1}+\varepsilon}. 
\]
Since, in the present example, \(\gamma=2\), this gives 
\[ 
\sum_{x\in E_R} \mu(x)H(x)v^{-\frac{1}{\sigma-1}}(x) \asymp R^{\frac{\gamma\sigma}{\sigma-1}+\varepsilon}. 
\]

We now construct a positive nontrivial solution of \eqref{Pb-ex2}. Set
\[
\delta=1+\frac{2}{\sigma-1}+\frac{\varepsilon}{2},
\qquad
\rho=\frac{\varepsilon(\sigma-1)}{2},
\]
and define
\[
f(a_n)=(n+n_0)^{-\delta},
\qquad n\ge0.
\]
Since
\[
D=1+\frac{2}{\sigma-1}+\varepsilon,
\]
we have
\[
\delta<D.
\]

For \(n\ge1\), by the definition of \(\Delta_H\), we have
\[
\begin{aligned}
\Delta_H f(a_n)
&=
\frac{1}{(n+n_0)^2(c_n+c_{n-1})}
\Big[
(n+n_0)(n+n_0+1)c_n\bigl(f(a_{n+1})-f(a_n)\bigr)\\
&\qquad\qquad\qquad\qquad
+(n+n_0)(n+n_0-1)c_{n-1}\bigl(f(a_{n-1})-f(a_n)\bigr)
\Big]\\
&=
\frac{(n+n_0)^{-\delta-1}}{c_n+c_{n-1}}
\Bigg[
c_n(n+n_0+1)
\left(
\left(\frac{n+n_0}{n+n_0+1}\right)^\delta-1
\right)\\
&\qquad\qquad\qquad
+c_{n-1}(n+n_0-1)
\left(
\left(\frac{n+n_0}{n+n_0-1}\right)^\delta-1
\right)
\Bigg].
\end{aligned}
\]
Using \(c_n=(n+1)^{D-1}\), we obtain
\[
\Delta_H f(a_n)
\sim
\frac{\delta(\delta-D)}{2}(n+n_0)^{-\delta-2}
\qquad\text{as }n\to\infty.
\]
Since \(\delta<D\), the coefficient \(\delta(\delta-D)/2\) is negative.
Hence, choosing \(n_0\) sufficiently large, there exists \(C_1>0\) such that
\begin{equation}\label{DeltaH-power-ex2}
\Delta_H f(a_n)
\le
-C_1(n+n_0)^{-\delta-2},
\qquad n\ge0.
\end{equation}
Choose \(A>0\) sufficiently small so that
\begin{equation}\label{A-small-ex2}
A^{\sigma-1}n_0^{-\rho}\le \frac{C_1}{2}.
\end{equation}

We define
\[
u(a_n)=A(n+n_0)^{1-\delta},
\qquad n\ge0.
\]
Then \(u>0\) in \(V\). Moreover,
\begin{equation}\label{Afn}
\frac{u(a_n)}{H(a_n)}
=
A(n+n_0)^{-\delta}=Af(a_n).
\end{equation}
Since \(\Delta H+wH=0\), the identity
\[
\Delta u+w u
=
H\Delta_H\left(\frac{u}{H}\right)
\]
holds. Hence
\begin{equation}\label{OKK}
\Delta u+w u+u^\sigma
=
H\left[
\Delta_H\left(\frac{u}{H}\right)
+
H^{\sigma-1}\left(\frac{u}{H}\right)^\sigma
\right].
\end{equation}
By \eqref{DeltaH-power-ex2} and \eqref{Afn},
\begin{equation}\label{hot1}
\Delta_H\left(\frac{u}{H}\right)(a_n)
=
A\Delta_H f(a_n)
\le
-C_1A(n+n_0)^{-\delta-2}.
\end{equation}
On the other hand,
\[
H^{\sigma-1}(a_n)\left(\frac{u(a_n)}{H(a_n)}\right)^\sigma
=
A^\sigma (n+n_0)^{\sigma-1-\delta\sigma}.
\]
Since
\[
\sigma-1-\delta\sigma=-\delta-2-\rho,
\]
we obtain
\[
H^{\sigma-1}(a_n)\left(\frac{u(a_n)}{H(a_n)}\right)^\sigma
=
A^\sigma (n+n_0)^{-\delta-2-\rho}
\le
A^\sigma n_0^{-\rho}(n+n_0)^{-\delta-2}.
\]
Using \eqref{A-small-ex2}, this gives
\begin{equation}\label{hot2}
H^{\sigma-1}(a_n)\left(\frac{u(a_n)}{H(a_n)}\right)^\sigma
\le
\frac{C_1}{2}A(n+n_0)^{-\delta-2}.
\end{equation}
Therefore, by \eqref{hot1} and \eqref{hot2},
\[
\Delta_H\left(\frac{u}{H}\right)(a_n)
+
H^{\sigma-1}(a_n)\left(\frac{u(a_n)}{H(a_n)}\right)^\sigma
\le0,
\qquad n\ge0.
\]
Since \(H>0\) in \(V\), it follows from \eqref{OKK} that
\[
\Delta u(a_n)+w(a_n)u(a_n)+u^\sigma(a_n)\le0,
\qquad n\ge0.
\]
Thus \(u\) is a positive nontrivial solution of \eqref{Pb-ex2}.
\end{example}

\bigskip

\noindent{\bf Author Contributions} M.J. and B.S. wrote the manuscript.

\noindent{\bf Funding} The first author is supported by the Ongoing Research Funding Program, (ORF-2026-57), King Saud University, Riyadh, Saudi Arabia.

\noindent{\bf Data Availability} No datasets were generated or analysed during the current study

\medskip 

\noindent{\bf Declarations}

\medskip 
\noindent{\bf Ethical Approval} Not applicable.

\noindent{\bf Conflict of interest}  On behalf of all authors, the corresponding author states that there is no conflict of interest.

\end{document}